\documentclass{imsart}

\RequirePackage[OT1]{fontenc}
\RequirePackage{amsthm,amsmath,epsfig}
\RequirePackage[colorlinks,citecolor=blue,urlcolor=blue]{hyperref}
\usepackage{a4wide,xcolor,eucal,enumerate,mathrsfs,graphics,caption,subfig,booktabs,verbatim,amssymb,dsfont}
\usepackage[numbers]{natbib}

\newcommand{\R}{\mathbb{R}}

\newcommand{\p}{\mathbb{P}}

\newcommand{\E}{\mathbb{E}}

\newcommand{\N}{\mathbb{N}}

\newcommand{\F}{\mathcal{F}}

\newcommand{\1}{\mathds{1}}

\newcommand{\dd}{\mathrm{d}}

\newcommand{\comR}[1]{{\color{red}#1}}

\startlocaldefs
\numberwithin{equation}{section}
\theoremstyle{plain}
\newtheorem{theo}{Theorem}[section]
\theoremstyle{remark}
\newtheorem{re}{Remark}[section]
\theoremstyle{proposition}
\newtheorem{prop}{Proposition}[section]
\theoremstyle{lemma}
\newtheorem{lemma}{Lemma}[section]
\endlocaldefs

\begin{document}

\begin{frontmatter}
\title{Estimation of the marginal expected shortfall under asymptotic independence}
\runtitle{ MES estimation under asymptotic independence}

\begin{aug}
\author{\fnms{Juan-Juan} \snm{Cai}\thanksref{e1}\ead[label=e1,mark]{j.j.cai@tudelft.nl}}
\and
\author{\fnms{Eni} \snm{Musta}\thanksref{e1}\ead[label=e2,mark]{e.musta@tudelft.nl}}


\address{Delft Institute of Applied Mathematics, Mekelweg 4, 2628 CD Delft, The Netherlands.
\printead{e1,e2}}


\runauthor{Cai and Musta}

\affiliation{Delft University of Technology}

\end{aug}

\begin{abstract}
We study the asymptotic behavior of the marginal expected shortfall when the two random variables are asymptotic independent but positive associated, which is modeled by the so-called tail dependent coefficient. We construct an estimator of the marginal expected shortfall {which is shown to be asymptotically normal. The finite sample performance of the estimator is investigated in a small simulation study.}
The method is also applied to estimate the expected amount of rainfall at a weather station given that there is a once every 100 years rainfall at another weather station nearby. 
\end{abstract}

\begin{keyword}
\kwd{marginal expected shortfall}
\kwd{asymptotic independence}
\kwd{tail dependent coefficient}
\end{keyword}

\end{frontmatter}

\section{Introduction}
Let $X$ and $Y$ denote two risk factors. The marginal expected shortfall (MES)  is defined as $\mathbb{E}[X|Y>Q_Y(1-p)]$, where $Q_Y$ is the quantile function of $Y$ and $p$ is a {\it small} probability. The name of MES originates from its application in finance as an important ingredient for constructing a systemic risk measure; see for instance \citep{Acharyaetal2012} and \citep{CaporinSantucci12}.  In actuarial science, this quantity is known as the multivariate extensions of tail conditional expectation (or, conditional tail expectation); see for instance \citep{CaiLi2005} and \citep{Cousin2014}.


Under the assumption that $X$ is in the  Fr\'echet domain of attraction, \citep{Caietal2015} has established the following asymptotic limit; see Proposition 1 in that paper. With $Q_X$  the quantile function of $X$, 
\begin{align}
\lim_{p\rightarrow 0}\frac{\E[X|Y>Q_Y(1-p)]}{Q_X(1-p)}=a \in[0,\infty) \label{eq: MESdep}
\end{align}
where $a>0$ if $X$ and $Y$ are asymptotic dependent and $a=0$ if they are asymptotic independent. Based on this result, an estimator for MES is established in \citep{Caietal2015}, which is not applicable for asymptotic independent data. It is the goal of this paper to study the asymptotic behavior of MES and to develop an estimation of MES for asymptotic independent data.

Under the framework of multivariate extreme value theory, there are various ways to describe asymptotic dependence, for instance by means of exponent measure, spectral measure or Pickands dependence functions, cf. Chapter 6 of  \citep{deHaanFerreira2006} and Chapter 8  of  \citep{Beirlantetal2004}.  However, these measures don't distinguish the relative strength of the extremal dependence for an asymptotic independent pair. The so-called {\it coefficient of tail dependence} introduced by \citep{LedfordTawn1996} is mostly used to  measure the strength of the extremal dependence for an asymptotic independent pair. In this paper,  we make use  of the coefficient of tail dependence, denoted as $\eta$ to model asymptotic independence. Namely, we assume that there exists an $\eta \in (0,1]$ such that the following limit exists and is positive:
\[
\lim_{p\rightarrow0} p^{-\frac{1}{\eta}}\p\left(X>Q_X(1-p) \text{ and  } Y>Q_Y(1-p)   \right).
\]
We are interested in the scenario that $\eta\in (1/2, 1) $, which corresponds to asymptotic independence but positive association of $X$ and $Y$. For this type of distributions, one has
$$
\p\left(X>Q_X(1-p) \text{ and  } Y>Q_Y(1-p)   \right)\gg \p\left(X>Q_X(1-p)\right) \p\left(Y>Q_Y(1-p)   \right),
$$
that is, the joint extremes of $(X, Y)$ happen much more often than those of a distribution with independent components of $X$ and $Y$. This gives an intuitive explanation that even if the pair are asymptotic independent, the extremal dependence can still be strong and thus needs to be accounted for. We also assume that $X$ is in the  Fr\'echet domain of attraction, so it has a heavy right tail. As our result shall reveal, the risk represented by MES can also be very big under the combination of positive association and  $X$ being heavy tailed, cf. Proposition \ref{prop:limit_theta_p}. Thus from the application point of view, it is very important to consider MES for such a model assumption.

This paper is organised as follows. Section \ref{sc: main results} contains the main theoretical results on the limit behaviour of $MES$ and the asymptotic normality of the proposed estimator of MES. The performance of the estimation method is illustrated by a simulation study in Section \ref{sc: simulation} and by an application to precipitation data in Section \ref{sc: application}. The proofs of the main theorems  are provided in Section \ref{sc: proof}.


\section{Main results} \label{sc: main results}
We first derive the asymptotic limit for MES as $p\rightarrow 0$, based on which we shall then construct an estimator for MES. Let $F_1$ and $F_2$ denote the marginal distribution functions of $X$ and $Y$, respectively.  As usual in extreme value analysis, it is more convenient to work with, in stead of the quantile function, the tail quantile defined as 
$U_j=\left(\frac{1}{1-F_j} \right)^\leftarrow,  j=1,2,$ where $\leftarrow$ denotes the left continuous inverse. Then MES can be written as 
\[
\E\left[X|Y>U_2(1/p)\right]=:\theta_p.
\]

We now present our model,  namely, assumptions on the tail distribution of $X$ and the extremal dependence of $X$ and $Y$. First, we assume that $X$ has a heavy right tail, that is, there exists a $\gamma_1>0$ such that
\begin{align}
\lim_{t\rightarrow \infty}\frac{U_1(tx)}{U_1(t)}=x^{\gamma_1}, \qquad x>0. \label{eq: rv_U1}
\end{align}
Second, we assume the positive association of $X$ and $Y$. Precisely, there exists an $\eta\in(1/2, 1]$ such that for all $(x, y) \in (0,\infty)^2$, the following limit exists
\begin{equation}
\lim_{t\to\infty}t^{\frac{1}{\eta}}\p\left(X>U_1(t/x),Y>U_2(t/y)\right)=:c(x,y)\in(0, \infty).  \label{eq: asymind}
\end{equation}
As a consequence, $c$ is a homogeneous function of order $\frac{1}{\eta}$. The condition of \eqref{eq: asymind} is also assumed in \citep{DDFH04} for estimating $\eta$.
Note that if $\eta=1$, it corresponds to $X$ and $Y$ being asymptotic dependent. For $\eta<1$, this condition is linked to the so-called hidden regular variation (cf.\ \citep{Resnick2002}) in the following way:
$$
\nu^*((x,\infty]\times (y,\infty])=c(x^{-1/\gamma_1}, y^{-1/\gamma_1}),
$$
where $\nu*$ is defined in (3) of \citep{Resnick2002}.

In order to obtain the limit result on $\theta_p$ for $p\rightarrow 0$, we need a second order strengthening condition of \eqref{eq: rv_U1}.  
\begin{description}
\item[A(1)] There exists $d\in(0,\infty)$ such that 
\[
\lim_{t\to\infty}\frac{U_1(t)}{t^{\gamma_1}}=d.
\]
\end{description}
We also need some technical conditions on the extremal dependence of $X$ and $Y$. For $t>0$, define
\begin{equation}
c_t(x,y)=t^{\frac{1}{\eta}}\p\left(X>U_1(t/x),Y>U_2(t/y)\right),\qquad 0<x,y<t.
\end{equation}
\begin{description}
\item[A(2)] There exists $\beta_1>\gamma_1$ such that $\lim_{t\to\infty}\sup_{x\leq 1}\left|c_t(x,1)-c(x,1)\right|x^{-\beta_1}=0$.
\item[A(3)] There exists $0<\beta_2<\gamma_1$ such that $\lim_{t\to\infty}\sup_{1<x<t}\left|c_t(x,1)-c(x,1)\right|x^{-\beta_2}=0$.
\end{description}

\begin{prop}
\label{prop:limit_theta_p}
Assume that $X$ takes values in $(0,\infty)$ and conditions A(1)-A(3) hold. If $-\frac{1}{\eta}+1+\gamma_1>0$, $\int_0^\infty c\left(x^{-\frac{1}{\gamma_1}},1\right)\,\dd x<\infty$ {and $x\mapsto c(x,1)$ is a continuous function}, then we have
\begin{equation}
\lim_{t\to\infty}\frac{\theta_{1/t}}{t^{-\frac{1}{\eta}+1}U_1(t)}=\int_0^\infty c\left(x^{-\frac{1}{\gamma_1}},1\right)\,\dd x.  \label{eq: limit_theta}
\end{equation}
\end{prop}

\begin{re} If $\eta=1$, then the result of \eqref{eq: limit_theta} coincide with the result of Proposition 1 in \citep{Caietal2015}.
\end{re}
Provided with a random sample $(X_1, Y_1),\ldots, (X_n, Y_n)$, we now construct an estimation of $\theta_p$, where $p=p(n)\rightarrow 0$, as $n\rightarrow\infty$. Propositions \ref{prop:limit_theta_p} suggests  the following approximation. With $t$ sufficiently large, 
\[
\theta_p\sim\left(\frac{1}{pt}\right)^{-\frac{1}{\eta}+1}\frac{U_1(1/p)}{U_1(t)}\theta_{\frac{1}{t}}\sim \left(\frac{1}{pt}\right)^{-\frac{1}{\eta}+1+\gamma_1}\theta_{\frac{1}{t}}.
\]
We choose $t=\frac{n}{k}$, where $k=k(n)$ is a sequence of integers such that $k\to\infty$ and $k/n\to 0$, as $n\to \infty$. Then, 
\[
\theta_p\sim \left(\frac{k}{pn}\right)^{-\frac{1}{\eta}+1+\gamma_1}\theta_{\frac{k}{n}}.
\]
From this extrapolation relation, the remaining task is to estimate $\eta$, $\gamma_1$ and $\theta_{\frac{k}{n}}$. There are well-known existing methods for estimating $\gamma_1$ and $\eta$; see Chapters 3 and 7 of \citep{deHaanFerreira2006}. For $\theta_{\frac{k}{n}}$,
we propose a nonparametric estimator given by
\begin{equation}
\hat{\theta}_{k/n}=\frac{1}{k}\sum_{i=1}^n X_i\1_{\{Y_i>Y_{n-k,n}\}}.  \label{eq: hat_thetank}
\end{equation}
Let $\hat\gamma_1$ and $\hat\eta$ denote estimators of $\gamma_1$ and $\eta$, respectively. We construct the following estimator for $\theta_p$:  
\begin{align}
\hat{\theta}_p=\hat{\theta}_{k/n}\left(\frac{k}{np}\right)^{-\frac{1}{\hat{\eta}}+1+\hat{\gamma_1}}. \label{eq: hat_thetap}
\end{align}

Next we prove the asymptotic normality of $\hat{\theta}_p$. The following  conditions will be needed. 
\begin{description}
\item[B(1)] $-\frac{1}{\eta}+1+\gamma_1>0$ and there exists $\bar{\delta}>0$ such that, for $\rho\in\{1,2,2+\bar{\delta}\}$, 
\[
\sup_{y\in[1/2,2]}\int_0^\infty c\left(x,y\right)\,\dd x^{-\rho\gamma_1}<\infty,~ \text{and} \quad \int_1^\infty c(x,y)^2\,\dd x^{-\gamma_1}<\infty.
\]
\item[B(2)] There exists $\beta_1>(2+\bar{\delta})\gamma_1$ and $\tau<0$ such that 
\[
  \sup_{\substack{x\leq1\\1/2\leq y\leq 2}}\left|c_t(x,y)-c(x,y)\right|x^{-\beta_1}=O(t^\tau).
\]
\item[B(3)] There exists $-\tau/(1-\gamma_1)<\beta_2<\gamma_1$ such that 
\[
  \sup_{\substack{1<x<t\\1/2\leq y\leq 2}}\left|c_t(x,y)-c(x,y)\right|x^{-\beta_2}=O(t^\tau)
\]
with the same $\tau$ as in B(2) .
\item[B(4)]  There exists $\rho_1<\frac{1}{2}-\frac{1}{2\eta}$ and a regularly varying function $A_1$ with index $\rho_1$ such that
\[
\sup_{x>1}\left|x^{-\gamma_1}\frac{U_1(tx)}{U_1(t)}-1 \right|=O\left\{ A_1(t)\right\}.
\]
\item[B(5)] As $n\rightarrow \infty$, $k=O(n^\alpha)$ for some $\alpha$ that satisfies the following condition
\[
1-\eta<\alpha<\min\left(1-\frac{\eta}{1+\eta\gamma_1\lambda}, 1+\frac{\eta}{1-2\eta-2\eta\gamma_1\lambda},
\frac{-\frac{1}{\eta}+1+2\tau+2\beta_2(1-\gamma_1)}{-\frac{1}{\eta}+2\tau+2\beta_2(1-\gamma_1)},1+\frac{1}{2\rho_1-1}\right)
\]
with some $\max\left(\beta_2,\frac{1-\eta}{\gamma_1\eta} \right)<\lambda<1$.
\item[B(6)] $\hat \gamma_1$ is such that $\sqrt{k}(\hat\gamma_1-\gamma_1)=O_p(1)$, and $\hat \eta$ is such that $\sqrt{k}(\hat\eta-\eta)=O_p(1)$. 
\end{description}
\begin{theo}
\label{theo:as_normality_theta_p}
Suppose  that $X$ takes values in $(0,\infty)$, {$\gamma_1\in(0,1)$} and Conditions B(1)-B(6) hold. Assume that $d_n=k/(np)\geq 1$ and $\lim_{n\to\infty}(n/k)^{-1/2\eta+1/2}\log (d_n)=0$. Then, as $n\to\infty$,
\[
\left(\frac{n}{k}\right)^{-\frac{1}{2\eta}+\frac{1}{2}}\sqrt{k}\left(\frac{\hat{\theta}_p}{\theta_p}-1 \right)\xrightarrow{d} N(0,\sigma^2)
\]
where
\[
\sigma^2=\left(\int_0^\infty c(x,1)\,\dd x^{-\gamma_1}\right)^{-2}.
\]
\end{theo}
{
\begin{re}
Note that the condition $\lim_{n\to\infty}(n/k)^{-1/2\eta+1/2}\log (d_n)=0$ implies that $\eta<1$. Moreover, from $\gamma_1<1$ and $-\frac{1}{\eta}+1+\gamma_1>0$ (see B(1)), it follows that $\eta>1/2$.
\end{re}
}

\section{Simulation Study} \label{sc: simulation}
In this section, we study the finite sample performance of our method. We apply our estimator given by \eqref{eq: hat_thetap} to data simulated from the following two distributions. We consider sample size $n=5000$, and $p=10/n, 1/n$ and $1/(10n)$.

Let $Z_1,\,Z_2$, and $Z_3$ be independent Pareto random variables with parameters 0.3, 0.4, and 0.4 


{\it Example 1.}  Define $X=Z_1\vee Z_2$ and $Y=Z_1\vee Z_3$.

{\it Example 2.}  Define
$
(X, Y)=B(Z_1, Z_1)+(1-B)(X_1, X_3),
$
where $B$ is a Bernoulli $(1/2)$ random variable independent of $Z_1,\,Z_2$, and $Z_3$.

For both distributions, $\gamma_1=0.4$, $\eta=\frac{3}{4}$ and $c(x,y)=d(x\wedge y)^{1/\eta}$, $0<x, y<\infty$, for some constant $d>0$.
And assumptions B(1)-B(7) are satisfied by both models with properly chosen $\bar{\delta}$, $\beta_1$, $\beta_2$,  $\tau$ and $\alpha$. To complete our estimator given by \eqref{eq: hat_thetap}, we use the Hill estimator for $\gamma_1$ and an estimator for $\eta$ proposed in \citep{DDFH04}. Let $k_1$ and $k_2$ be two intermediate sequences.  Define
\[
\hat{\gamma}_1=\frac{1}{k_1}\sum_{i=1}^{k_1}\log(X_{n-i+1,n})-\log(X_{n-k_1,n}),
\]
and
\begin{equation}
\hat{\eta}=\frac{1}{k_2}\sum_{i=1}^{k_2}\log\frac{T^{(n)}_{n,n-i+1}}{T^{(n)}_{n,n-k_2}},  \label{eq: hat_eta}
\end{equation}
where $ T^{(n)}_i=\frac{n+1}{n+1-R^X_i}\wedge\frac{n+1}{n+1-R^Y_i}$ with $R_i^X$  and $R_i^Y$ denoting the ranks of $X_i$ and $Y_i$ in their respective samples. 

For $p=10/n$, we compare our estimator with the nonparametric estimator, namely,
\[
\hat{\theta}_{emp}=\frac{1}{10}\sum_{i=1}^n X_i\1_{\{Y_i>Y_{n-10,n}\}},
\]
which is obtained by letting $k/n=p$ in \eqref{eq: hat_thetank}.

For each estimator, we compute the relative error defined as
$bias_p=\frac{1}{m}\sum_{i=1}^m \frac{\hat \theta_{p,i}}{\theta_p}-1$, where $\hat \theta_{p, i}$ is an estimate based on the $i$-th sample. A relative error for $\hat \theta_{emp}$ is computed in the same way, denoted as $bias_{emp}$.  Figure \ref{Fig: boxplot} shows the relative errors obtained by generating $m=500$ samples for each scenario. From the boxplots, for the situation where the empirical estimator is applicable, that is $p=10/p$, our estimator has a slightly larger bias but a smaller variance. As $p$ becomes smaller, the empirical estimator is not applicable, yet our estimator still has decent performance with growing variance. 

\begin{figure}
\begin{center}
\captionsetup{format=hang}
\includegraphics[width=\textwidth]{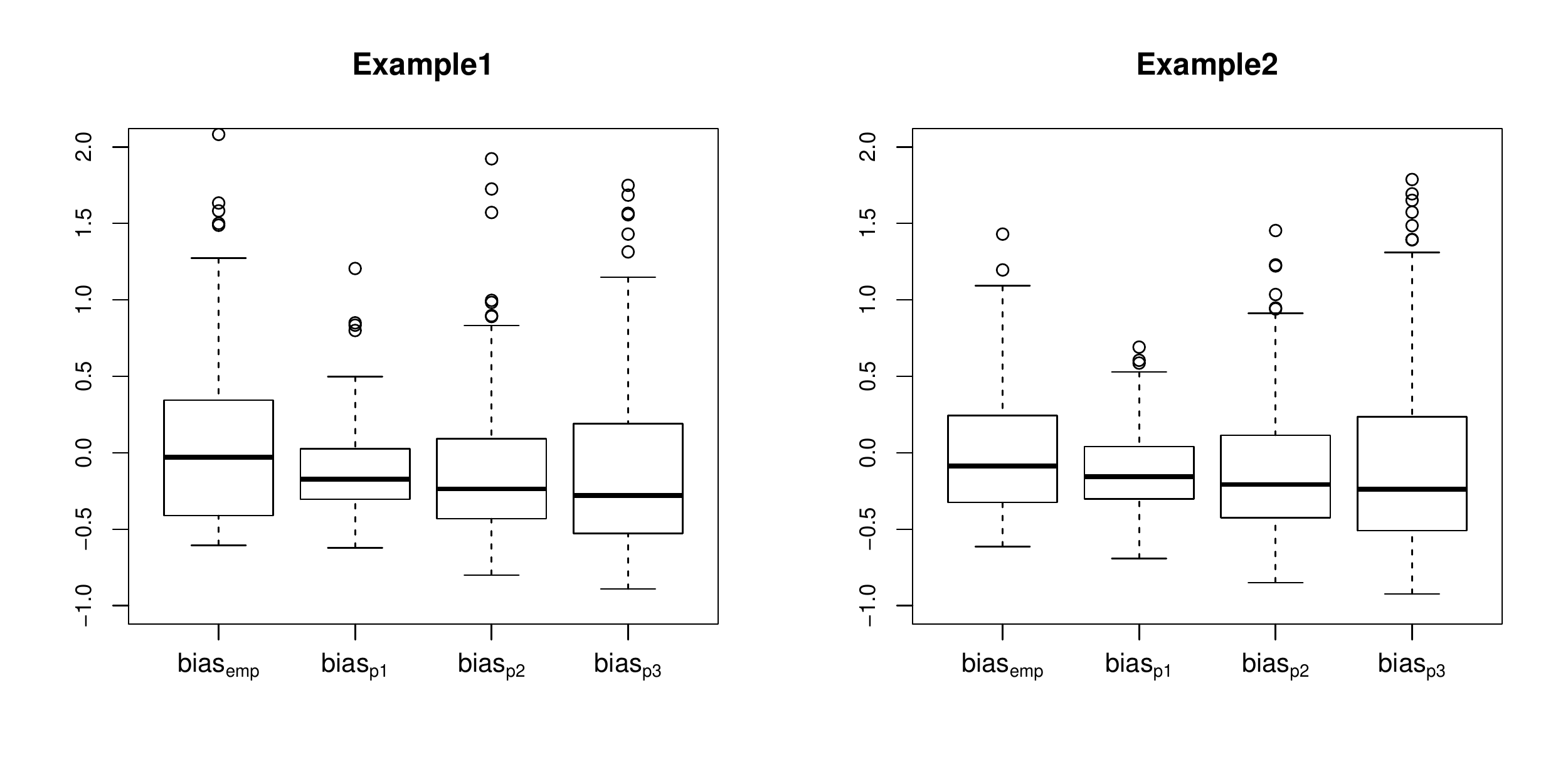}
\caption{The relative errors of the estimators with $n=5000$, $p1=10/n$, $p2=1/n$, $p3=1/(10n)$, and $k=k_1=k_2=200$, for both examples.}
\label{Fig: boxplot}
\end{center}
\end{figure}

\section{Application} \label{sc: application}
We apply our estimation to daily precipitation data from two weather stations in the Netherlands, namely Cabauw and Rotterdam. The distance between these two stations is about 32 km. The station Cabauw is close to the Lek river while the station Rotterdam is close to the river Nieuwe Maas, which is the continuation of Lek.  Heavy rainfall at both stations might lead to a severe flood in this region. Thus, the expected amount of rainfall in Cabauw given a heavy rainfall in Rotterdam is an important risk measure for the hydrology safety control.  We estimate this quantity based on the data from August 1st, 1990 to December 31, 2016. After removing the missing values, there are in total 9605 observations.  There is open access to the data at \url{http://projects.knmi.nl/klimatologie/uurgegevens/selectie.cgi}.   

Let $X$ be the daily rainfall at Cabauw and $Y$ be the daily rainfall at Rotterdam. Before applying our estimation, we shall look at the sign of the extreme value index of $X$ and the extremal dependence of $X$ and $Y$. From the Hill estimates of $\gamma$ as shown in  right panel of Figure \ref{Fig: evi_eta_cab_ro}, we conclude that $\gamma>0$, which is in line with the existing literature.  For instance,  \citep{Buishandetal2008} obtains 0.1082 as the estimate of $\gamma$ for the daily rainfall in the Netherlands   and  \citep{ColesTawn1996} reports 0.066 as the estimate of $\gamma$ for the daily rainfall in the southwest of England.

\begin{figure}
\begin{center}
\captionsetup{format=hang}
\includegraphics[scale=0.4]{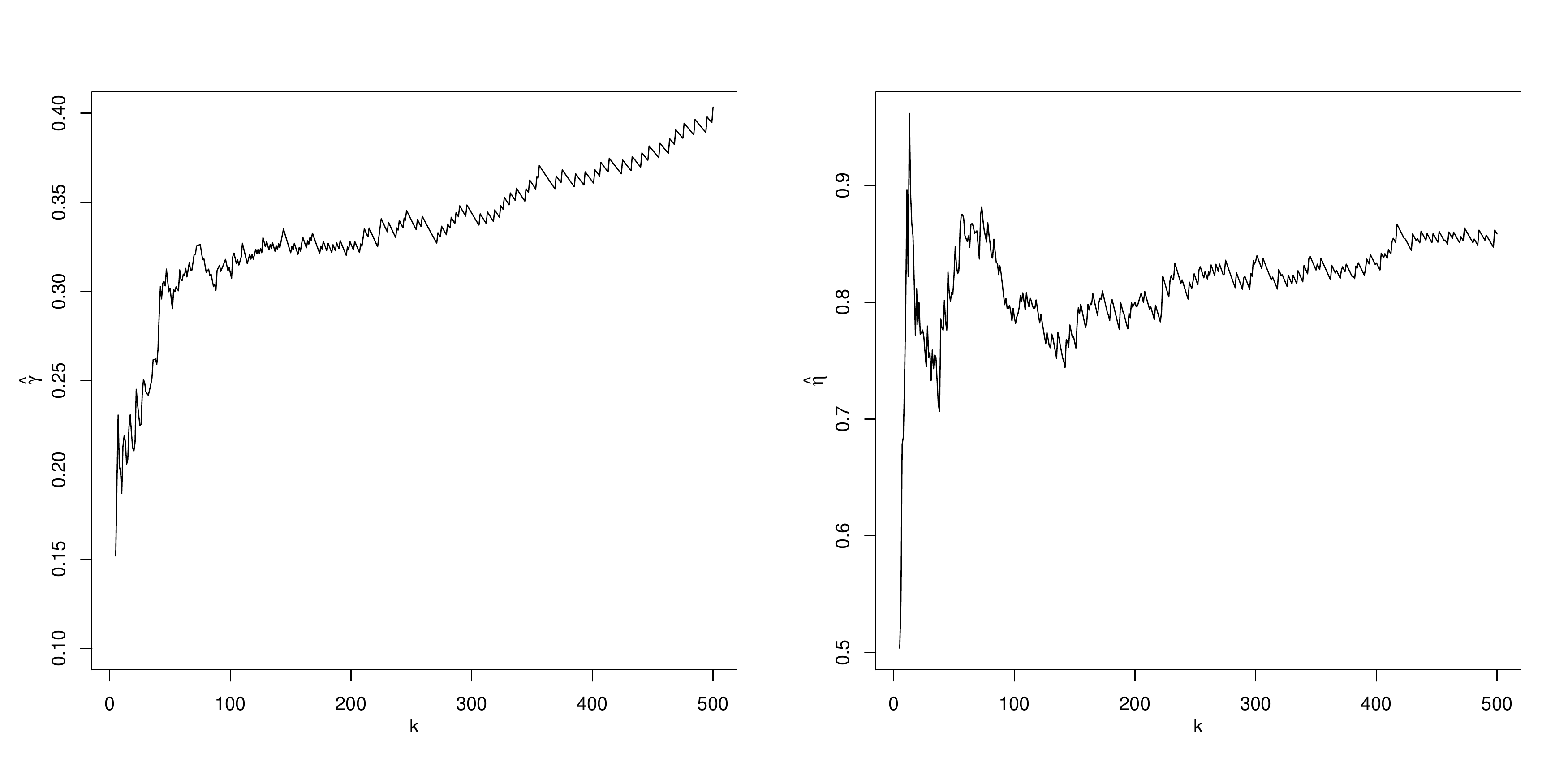}
\caption{The Hill estimates of $\gamma_1$ for the daily precipitation at Cabauw (left panel) and the estimates of $\eta$ for the daily precipitations at Cabauw and Rotterdam (right panel).}
\label{Fig: evi_eta_cab_ro}
\end{center}
\end{figure}

 Next, we compute the Hill estimator of $\eta$ given by \eqref{eq: hat_eta}. The estimates are above 0.5 as shown in the right panel of Figure \ref{Fig: evi_eta_cab_ro}.

Finally, we apply our estimator to answer the following question. Provided that the amount of rainfall in Rotterdam exceeds the  $M$-year return level, what is the expected amount of rainfall in Cabauw, respectively, for $M=50$ and 100? Let $R_M$ denote the $M$-year return level. \citep{Coles2001} gives the definition of $R_M$ as the level expected to be exceeded once every $M$ years.  As we consider the daily precipitation, $R_M=U_2(365M)$. 

 Choosing $k_1=k_2=200$, we obtain the following estimates of $\gamma$ and $\eta$: $\hat \gamma_1=0.326$ and $\hat\eta=0.835$. 
Figure \ref{Fig: theta} plots the estimates of $\theta_p$ against $k$, from which we conclude $k=50$ lying in the interval where the estimates are stable. We thus report the following estimates of $\hat \theta_{p}$: 41.6 mm for $M=50$ and 45.5 mm for $M=100$.

\begin{figure}
\begin{center}
\captionsetup{format=hang}
\includegraphics[scale=0.6]{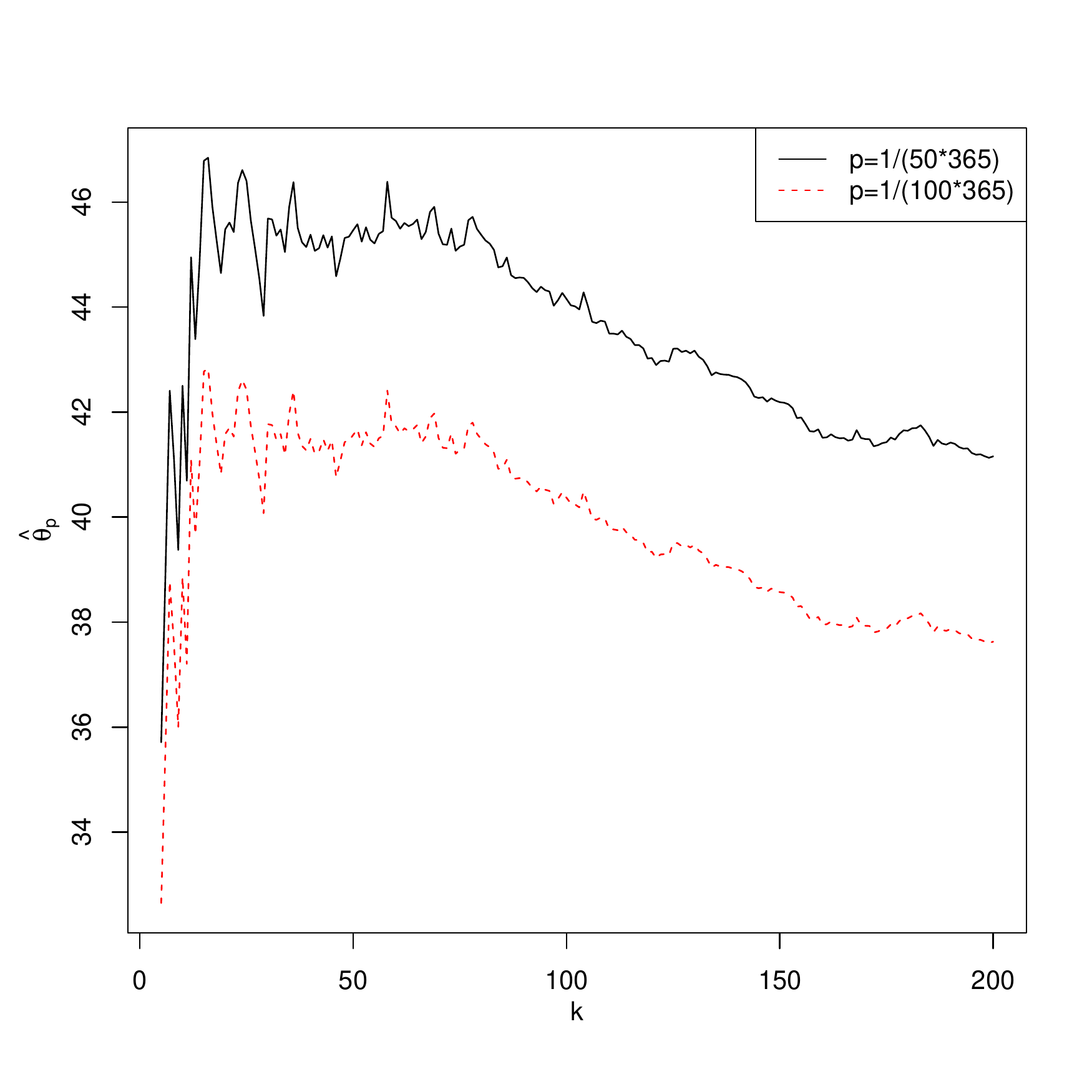}
\caption{The estimates of the conditional expected rainfall at Cabauw given that the rainfall at Rotterdam exceed the 50 (red) and 100 (black) year return level. }
\label{Fig: theta}
\end{center}
\end{figure}


\section{Proofs} \label{sc: proof}
\begin{proof}[Proof of Proposition \ref{prop:limit_theta_p}]

We recall that for any positive random variable $Z$,  the expectation can be written as
\begin{equation*}
\label{eqn:expectation}
\E[Z]=\int_0^\infty \p(Z>x)\,\dd x.
\end{equation*}
Then by definition of $\theta_p$ and a change of variable, we have
\begin{equation}
\label{eq: thetacn}
\begin{split}
\theta_{1/t}&=\int_0^\infty t\p(X>x, Y>U_2(t))\,\dd x\\
&=t^{-\frac{1}{\eta}+1}U_1(t)\int_0^\infty t^{\frac{1}{\eta}}\p(X>xU_1(t), Y>U_2(t))\,\dd x.
\end{split}
\end{equation}
Define $f_t(x):=t^{\frac{1}{\eta}}\p(X>xU_1(t), Y>U_2(t))$, $x>0$. Then
\begin{equation*}
\frac{\theta_{1/t}}{t^{-\frac{1}{\eta}+1}U_1(t)}=\int_0^\infty f_t(x)\,\dd x. 
\end{equation*}
For any fixed $x$, {by~\eqref{eq: asymind} and the continuity of the function $x\mapsto c(x,1)$, we have }
\begin{equation*}
\lim_{t\rightarrow\infty}f_t(x)=c\left(x^{-\frac{1}{\gamma_1}},1\right).  
\end{equation*}
We shall apply  the generalized dominated convergence theorem to validate that $\lim_{t\rightarrow\infty}\int_0^\infty f_t(x)\,\dd x =\int_0^\infty c\left(x^{-\frac{1}{\gamma_1}},1\right)\,\dd x. $

By assumption A(1), for any $\epsilon>0$, there exists $t_0$ such that 
\[
\left|\frac{U_1(t)}{t^{\gamma_1}}-d\right|<\epsilon,\qquad\text{for all }t>t_0.
\]
Hence, for $c_1=(d+\epsilon)/(d-\epsilon)$ and $x>c_1(t_0/t)^{\gamma_1}$, we get
\[
\frac{U_1(t)x}{U_1\left(t(x/c_1)^{\frac{1}{\gamma_1}}\right)}=\frac{U_1(t)/t^{\gamma_1}}{U_1\left(t(x/c_1)^{\frac{1}{\gamma_1}}\right)/(t^{\gamma_1}x/c_1)}c_1
>\frac{d-\epsilon}{d+\epsilon}c_1=1.
\]
Consequently, for $x>c_1(t_0/t)^{\gamma_1}$, 
\begin{equation*}
f_t(x)\leq t^{\frac{1}{\eta}}\p\left(X>U_1\left(t(x/c_1)^{\frac{1}{\gamma_1}}\right), Y>U_2(t)\right)=c_t\left((x/c_1)^{-\frac{1}{\gamma_1}},1\right).
\end{equation*}
On the other hand, for $0<x\leq c_1(t_0/t)^{\gamma_1}$, $f_t(x)\leq t^{\frac{1}{\eta}-1}$. Define
\[ g_t(x): = \left\{ \begin{array}{ll}
        c_t\left((x/c_1)^{-\frac{1}{\gamma_1}},1\right) & \mbox{if $x > c_1(t_0/t)^{\gamma_1}$};\\
        t^{\frac{1}{\eta}-1} & \mbox{\text{otherwise}}.\end{array} \right. \] 
Then $f_t(x)\leq g_t(x)$. By generalized dominated convergence theorem, it is then sufficient to prove that 
\[
\lim_{t\rightarrow\infty}\int_0^\infty g_t(x)\,\dd x =\int_0^\infty \lim_{t\rightarrow\infty} g_t(x)\,\dd x=\int_0^\infty c\left((x/c_1)^{-\frac{1}{\gamma_1}},1\right)\,\dd x. 
\]
Observe that 
\begin{align*}
\int_0^\infty g_t(x)\,\dd x=&\int_0^{c_1(t_0/t)^{\gamma_1}}  t^{\frac{1}{\eta}-1}\,\dd x+c_1\int_{(t_0/t)^{\gamma_1}}^{\infty} c_t\left(x^{-\frac{1}{\gamma_1}},1\right)\,\dd x \\
=&c_1t^{\gamma_1}_0t^{\frac{1}{\eta}-1-\gamma_1}+{c_1\int_{(t_0/t)^{\gamma_1}}^{\infty} c_t\left(x^{-\frac{1}{\gamma_1}},1\right)\,\dd x}\\
\to &0+{c_1 \int_0^\infty c\left(x^{-\frac{1}{\gamma_1}},1\right)\,\dd x,}
\end{align*}
as $t\rightarrow \infty$. The last convergence follows from that $\frac{1}{\eta}-1-\gamma_1<0$, {$\int_0^{(t_0/t)^{\gamma_1}} c\left(x^{-\frac{1}{\gamma_1}},1\right)\,\dd x\to 0$,}
and the fact that
\begin{align*}
&\left|\int_{{(t_0/t)^{\gamma_1}}}^{\infty} c_t\left(x^{-\frac{1}{\gamma_1}},1\right)\,\dd x -\int_{{(t_0/t)^{\gamma_1}}}^{\infty} c\left(x^{-\frac{1}{\gamma_1}},1\right)\,\dd x\right|\\
\leq& \int_{{(t_0/t)^{\gamma_1}}}^1 \left|c_t\left(x^{-\frac{1}{\gamma_1}},1\right)-c\left(x^{-\frac{1}{\gamma_1}},1\right)\right|\,\dd x+\int_1^\infty \left|c_t\left(x^{-\frac{1}{\gamma_1}},1\right)-c\left(x^{-\frac{1}{\gamma_1}},1\right)\right|\,\dd x\\
=&o(1)\int_{{(t_0/t)^{\gamma_1}}}^1 x^{-\beta_2/\gamma_1}\,\dd x +o(1)\int_1^\infty x^{-\beta_1/\gamma_1}\,\dd x\to 0,
\end{align*}
by Assumptions A(2) and A(3).

\end{proof}

Through out the proof section, we denote the convergence speed in Theorem \ref{theo:as_normality_theta_p} by
\begin{equation}
T_n=\sqrt{k}\left(\frac{n}{k}\right)^{-\frac{1}{2\eta}+\frac{1}{2}}.  \label{eq: def_Tn}
\end{equation}
From Assumption B(5), $T_n\to \infty$, as ${n\rightarrow \infty}$. By construction, the asymptotic normality of $\hat{\theta_p}$  depends on the asymptotic behaviour of $\hat{\theta}_{k/n}$, which is given in Proposition \ref{prop:as.distr_theta}.

\begin{prop}
\label{prop:as.distr_theta}
Under the assumptions of Theorem \ref{theo:as_normality_theta_p}, it holds
\[
\frac{T_n\left(\frac{n}{k}\right)^{\frac{1}{\eta}-1}}{U_1(n/k)}\left(\hat{\theta}_{\frac{k}{n}}-\theta_{\frac{k}{n}} \right)\xrightarrow{d} N(0, \sigma_1^2),
\]
where $\sigma_1^2=-\int_0^\infty c(x,1)\,\dd x^{-2\gamma_1}$.
\end{prop}

The proof of Proposition \ref{prop:as.distr_theta} is postponed to the Appendix.

\begin{proof}[Proof of Theorem \ref{theo:as_normality_theta_p}]
Recall that $d_n=\frac{k}{np}$. By the definition of $\hat{\theta_p}$, we make the following decomposition
\[
\begin{split}
\frac{\hat{\theta_p}}{\theta_p}
=\frac{d_n^{-\frac{1}{\hat{\eta}}+1+\hat{\gamma_1}}\hat{\theta}_{k/n}}{\theta_p}
&=d_n^{\hat{\gamma}_1-\gamma_1}d_n^{\frac{1}{\eta}-\frac{1}{\hat{\eta}}}\frac{\hat{\theta}_{k/n}}{\theta_{k/n}}\frac{d_n^{-\frac{1}{\eta}+1+\gamma_1}\theta_{k/n}}{\theta_p}\\
&=:I_1\cdot I_2\cdot I_3\cdot I_4.
\end{split}
\]
We shall show that these four terms all converges to unity at certain rates. First,  from the assumption that $\sqrt{k}(\hat{\gamma_1}-\gamma_1)=O_P(1)$, it follows that 
\[
\begin{split}
I_1-1&=e^{(\hat{\gamma_1}-\gamma_1)\log d_n}-1=(\hat{\gamma_1}-\gamma_1)\log d_n+o\left((\hat{\gamma_1}-\gamma_1)\log d_n \right)
=O_P\left(\frac{\log d_n}{\sqrt{k}} \right)=o_p\left(\frac{1}{T_n} \right).
\end{split}
\]
 In the last equality,  we used  the assumption that $\lim_{n\to\infty}(n/k)^{-1/2\eta+1/2}\log d_n=0$. Recall that $T_n$ is defined in \eqref{eq: def_Tn}.

In the same way, we get $I_2-1=o_p\left(\frac{1}{T_n} \right)$.

Combining Propositions \ref{prop:as.distr_theta} and \ref{prop:limit_theta_p}, we derive that 
\begin{align*}
T_n(I_3-1)=\frac{T_n}{{\theta}_{k/n}}\left(\hat{\theta}_{k/n}-{\theta}_{k/n}\right)
=&\frac{T_n\left(\frac{n}{k}\right)^{\frac{1}{\eta}-1}}{U_1(n/k)}\left(\hat{\theta}_{k/n}-{\theta}_{k/n}\right)
\cdot\frac{U_1(n/k)\left(\frac{n}{k}\right)^{-\frac{1}{\eta}+1}}{{\theta}_{k/n}}\\
\overset{P}{\to}& \left(\int_0^\infty c\left(x^{-\frac{1}{\gamma_1}},1\right)\,\dd x \right)^{-1} N(0, \sigma_1^2).
\end{align*}
That is,  $T_n(I_3-1) \overset{P}{\to}\Gamma_1$,
where $\Gamma_1$ is a normal distribution with mean zero and variance, $\sigma^2=-\int_0^\infty c(x,1)\,\dd x^{-2\gamma_1}\left(\int_0^\infty c\left(x^{-\frac{1}{\gamma_1}},1\right)\,\dd x \right)^{-2}$, which is the limit distribution in Theorem  \ref{theo:as_normality_theta_p}.

 Then we deal with the last term, $I_4$. Here we need a  rate for the convergence in Proposition \ref{prop:limit_theta_p}. Continuing with \eqref{eq: thetacn},
\begin{equation}
\label{eq: thetank}
\begin{split}
\frac{\theta_{k/n}}{U_1(n/k)\left(n/k \right)^{-\frac{1}{\eta}+1}}
=&\int_0^\infty \left(\frac{n}{k}\right)^{\frac{1}{\eta}}\p(X>xU_1(n/k), Y>U_2(n/k))\,\dd x\\
=&-\int_0^\infty  c_{\frac{n}{k}}(s_n(x), 1) \,\dd x^{-\gamma_1}, 
\end{split}
\end{equation}
with 
\begin{equation}
\label{eq: sn}
s_n(x)=\frac{n}{k}\left[1-F_1\left(U_1(n/k)x^{-\gamma_1}\right)\right].
\end{equation}
 By the regular variation of $1-F_1$, we have $\lim_{n\rightarrow\infty}s_n(x)=x$, for any $x>0$. By  Lemma \ref{le:properties} (iii) and (v) in the Appendix, we have that 
\begin{equation}
\label{eq: diffcnc}
\int_0^\infty  c_{\frac{n}{k}}(s_n(x), 1) \,\dd x^{-\gamma_1}=\int_0^\infty c(x,1)\,\dd x^{-\gamma_1}+o\left(\frac{1}{T_n}\right).
\end{equation}
It follows from Assumptions B(4) and B(5) that
\[
\begin{split}
\frac{U_1(1/p)}{U_1(n/k)d_n^{\gamma_1}}-1=O(A(n/k))=o\left(\frac{1}{\sqrt{k}} \right).
\end{split}
\]
Combing this result with \eqref{eq: thetank} and \eqref{eq: diffcnc} leads to
\[
\begin{split}
I_4&=\frac{\theta_{k/n}}{U_1(n/k)\left(n/k \right)^{-\frac{1}{\eta}+1}}\frac{U_1(1/p)\left(1/p \right)^{-\frac{1}{\eta}+1}}{\theta_p}\frac{U_1(n/k)}{U_1(1/p)}\left(\frac{k}{np}\right)^{\gamma_1}\\
&=\left(\int_0^\infty c(x,1)\,\dd x^{-\gamma_1}+o\left(\frac{1}{T_n}\right)\right)\left(\int_0^\infty c(x,1)\,\dd x^{-\gamma_1}+o\left(\frac{1}{T_n}  \right)\right)^{-1}\left( 1+o\left(\frac{1}{\sqrt{k}} \right)\right)^{-1}\\
&=1+o\left(\frac{1}{T_n}\right).
\end{split}
\]

Thus, we obtain
\[
\begin{split}
\frac{\hat{\theta}_p}{\theta_p}-1&=I_1I_2I_3I_4-1\\
&=\left[1+o_p\left(\frac{1}{T_n}\right)  \right]^2 \left[1+\frac{\Gamma_1}{T_n}+o_p\left(\frac{1}{T_n}\right) \right] \left[1+o\left(\frac{1}{T_n}\right) \right] -1\\
&=\frac{\Gamma_1}{T_n}+o_p\left(\frac{1}{T_n}\right) .
\end{split}
\]
The proof is completed.

\end{proof}

\section{Appendix: Proof of Proposition \ref{prop:as.distr_theta}}
 In this section, all the limit processes involved are defined in the same probability space via the Skorohod construction, i.e. they are only equal in distribution to the original processes.  If we define 
\[
e_n=\frac{n}{k}\left(1-F_2(Y_{n-k,n})\right),
\]
 we have
\[
\hat{\theta}_{k/n}=\frac{1}{k}\sum_{i=1}^n X_i\1_{\{Y_i>U_2(n/ke_n)\}}.
\]
Note that $e_n\xrightarrow{\p}1$ because $1-F_2(Y_{n-k,n})$ is the $(k+1)$th order statistics of a random sample from the standard uniform distribution.
 
We first investigate the asymptotic behavior of 
\[
\tilde{\theta}_{\frac{ky}{n}}=\frac{1}{ky}\sum_{i=1}^n X_i\1_{\left\{Y_i>U_2\left(\frac{n}{ky}\right)\right\}}
\]
as a random process for $y\in[1/2,2]$.  

Let $W(y)$ denote a mean zero Gaussian process on $[1/2, 2]$ with covariance structure  
\[
\E[W(y_1)W(y_2)]=-\frac{{1}}{y_1y_2}\int_0^\infty c(x,y_1\wedge y_2)\,\dd x^{-2\gamma_1},\quad\quad \quad y_1, y_2\in[1/2, 2].
\]

\begin{prop}
\label{prop:limit_process}
Suppose conditions B(1)-B(5) hold. Let $S_n=\left( \frac{n}{k}\right)^{\frac{1}{2\eta}-\frac{1}{2}}\sqrt{k}$. Then as $n\to\infty$,
\begin{equation}
\left\{\frac{S_n}{U_1(n/k)}\left(\tilde{\theta}_{\frac{ky}{n}}-\theta_{\frac{ky}{n}}\right)\right\}_{y\in[1/2, 2]}
\overset{d}{\to}\left\{W(y)\right\}_{y\in[1/2, 2]}. \label{eq: limit_process}
\end{equation}
{The convergence of the process holds in distribution in the Skorokhod space $D([1/2,2])$.}
\end{prop}
Before proving Proposition \ref{prop:limit_process}, we first show two lemmas. The first lemma states some 
some properties of  the functions $c_t(x,y)$ and $c(x,y)$, that will be used frequently in the proof. The second lemma is established to compute the covariance of the limit process in Proposition \ref{prop:limit_process}.
\begin{lemma}
\label{le:properties}
\begin{description}
\item[(i)] If $\int_0^\infty c(x,y)\,\dd x^{-2\gamma_1}<\infty $, then  
\[
\int_0^\infty  \int_0^\infty c(x_1\wedge x_2,y)\,\dd x_1^{-\gamma_1}\,\dd x_2^{-\gamma_1}=\int_0^\infty c(x,y)\,\dd x^{-2\gamma_1}.
\]
\item[(ii)] The function $y\mapsto \int_0^\infty c(x,y)\,\dd x^{-\gamma_1}$ is Lipschitz, i.e. there exists $C>0$ such that for each $y_1,y_2\in[1/2,2]$, 
\[
\left|\int_0^\infty c(x,y_1)\,\dd x^{-\gamma_1}-\int_0^\infty c(x,y_2)\,\dd x^{-\gamma_1}\right|\leq C|y_1-y_2|.
\]
\item[(iii)] Assumptions B(1), B(3)-B(5) imply that 
\[
\sup_{y\in[1/2,2]}T_n\left|\int_0^\infty c(x,y)\,\dd x^{-\gamma_1}-\int_0^\infty c(s_n(x),y)\,\dd x^{-\gamma_1}\right|\to 0.
\]
\item[(iv)] If Assumptions B(2)-B(3) hold, then, for $\rho=1,2,2+\bar{\delta}$, 
\begin{equation}
\label{eqn:limit_c_n}
\left|\int_0^\infty c_{\frac{n}{k}}{\left(x\wedge \frac{n}{k},y\right)}\,\dd x^{-\rho\gamma_1}-\int_0^\infty c(x,y)\,\dd x^{-\rho\gamma_1}\right|\to 0.
\end{equation}

\item[(v)] If Assumptions B(2), B(3) and B(5) hold, then,
\begin{equation}
\label{eq: cnc1}
\sup_{y\in[1/2,2]}{T_n}\left|\int_0^\infty c_{\frac{n}{k}}{\left(x\wedge \frac{n}{k},y\right)\,\dd x^{-\gamma_1} -\int_0^\infty c(x, y) \,\dd x^{-\gamma_1}} \right|\to 0,
\end{equation}
and,
\begin{equation}
\label{eq: cnc2}
\sup_{y\in[1/2,2]}{T_n}\left|\int_0^\infty c_{\frac{n}{k}}(s_n(x),y){\,\dd x^{-\gamma_1}-\int_0^\infty c(s_n(x),y) \,\dd x^{-\gamma_1} }\right|\to 0,
\end{equation}.
\end{description}
\end{lemma}
\begin{proof}
The first statement follows from simple transformations of the integral. Indeed we have
\[
\begin{split}
&\int_0^\infty \int_0^\infty c(x_1\wedge x_2,y)\,\dd x_1^{-\gamma_1}\,\dd x_2^{-\gamma_1}\\
&=\int_0^\infty\int_{x_1}^\infty c(x_1,y)\,\dd x_2^{-\gamma_1}\,\dd x_1^{-\gamma_1}+\int_0^\infty\int_0^{x_1} c(x_2,y)\,\dd x_2^{-\gamma_1}\,\dd x_1^{-\gamma_1}\\
&=2\int_0^\infty x_1^{-\gamma_1} c(x_1,y)\,\dd x_1^{-\gamma_1}\\
&=\int_0^\infty c(x,y)\,\dd x^{-2\gamma_1}.
\end{split}
\] 
By the homogeneity property of $c(x,y)$: 
$
c(kx,ky)=k^{\frac{1}{\eta}}c(x,y),
$
 we have
\[
\begin{split}
&\left|\int_0^\infty c(x,y_1)\,\dd x^{-\gamma_1}-\int_0^\infty c(x,y_2)\,\dd x^{-\gamma_1}\right|\\
&=\left|\int_0^\infty y_1^{1/\eta}c\left(\frac{x}{y_1},1\right)\,\dd x^{-\gamma_1}-\int_0^\infty y_2^{1/\eta}c\left(\frac{x}{y_2},1\right)\,\dd x^{-\gamma_1}\right|\\
&=\left|y_1^{1/\eta-2\gamma_1}\int_0^\infty c\left(x,1\right)\,\dd x^{-\gamma_1}-y_2^{1/\eta-2\gamma_1}\int_0^\infty c\left(x,1\right)\,\dd x^{-\gamma_1}\right|\\
&\leq K|y_1-y_2|.
\end{split}
\]

(iii) Let $l_n=\left(\frac{n}{k}\right)^{\lambda}$ with $\lambda$ as in Assumption B(5). We start by writing
\begin{equation}
\label{eqn:iii}
\begin{split}
&\sup_{y\in[1/2,2]}T_n\left|\int_0^\infty c(x,y)\,\dd x^{-\gamma_1}-\int_0^\infty c(s_n(x),y)\,\dd x^{-\gamma_1}\right|\\
&\leq \sup_{y\in[1/2,2]}T_n\left|\int_0^{l_n} [c(x,y)-c(s_n(x),y)]\,\dd x^{-\gamma_1}\right|\\
&\quad+\sup_{y\in[1/2,2]}T_n\left\{\left|\int_{l_n}^\infty c(x,y)\,\dd x^{-\gamma_1}\right|+\left|\int_{l_n}^\infty c(s_n(x),y)\,\dd x^{-\gamma_1}\right|\right\}.
\end{split}
\end{equation} First we deal with the first term in the right hand side of~\eqref{eqn:iii}. By the homogeneity property of $c(x,y)$, we have
\[
|c(x_1,y)-c(x_2,y)|\leq \left|\left(\frac{x_2}{x_1}\right)^{1/\eta}-1\right|c(x_1,y)
\]
It follows that,
\[
\begin{split}
&\sup_{y\in[1/2,2]}T_n\left|\int_0^{l_n} [c(x,y)-c(s_n(x),y)]\,\dd x^{-\gamma_1}\right|\\
&\leq \sup_{y\in[1/2,2]}T_n\left|\int_0^{l_n} \left|\left(\frac{s_n(x)}{x}\right)^{1/\eta}-1\right|c(x,y)\,\dd x^{-\gamma_1}\right|. 
\end{split}
\]
Note that, for any $\epsilon_0>0$, for sufficiently large $n$ and $x<l_n$, (see \citep{Cai}, page 85)
\[
\left|\frac{s_n(x)/x-1}{A_1(n/k)}-\frac{x^{-\rho_1}-1}{\gamma_1\rho_1} \right|\leq x^{-\rho_1}\max(x^{\epsilon_0},x^{-\epsilon_0}).
\]
This implies that
\[
\left|\frac{s_n(x)}{x}-1\right|\leq |A_1(n/k)|\left\{\left|\frac{x^{-\rho_1}-1}{\gamma_1\rho_1}\right|+ x^{-\rho_1}\max(x^{\epsilon_0},x^{-\epsilon_0})\right\}.
\]
Since, for $\epsilon_0<-\rho_1(1-\lambda)/\lambda$,
\[
|A_1(n/k)|\left\{\left|\frac{x^{-\rho_1}-1}{\gamma_1\rho_1}\right|+ x^{-\rho_1}\max(x^{\epsilon_0},x^{-\epsilon_0})\right\}=o(1)
\]
by a Taylor expansion, we obtain
\[
\begin{split}
\left|\frac{s_n(x)}{x}-1\right|&=|A_1(n/k)|\left\{\left|\frac{x^{-\rho_1}-1}{\gamma_1\rho_1}\right|+ x^{-\rho_1}\max(x^{\epsilon_0},x^{-\epsilon_0})\right\}\\
&\quad+o\left(A_1(n/k)\left\{\left|\frac{x^{-\rho_1}-1}{\gamma_1\rho_1}\right|+ x^{-\rho_1}\max(x^{\epsilon_0},x^{-\epsilon_0})\right\}\right).
\end{split}
\]
Consequently,
\begin{equation}
\label{eqn:iii_1}
\begin{split}
&\sup_{y\in[1/2,2]}T_n\left|\int_0^{l_n} \left|\left(\frac{s_n(x)}{x}\right)^{1/\eta}-1\right|c(x,y)\,\dd x^{-\gamma_1}\right|\\
&\leq C\sup_{y\in[1/2,2]}T_n|A(n/k)|\left|\int_0^{l_n} x^{-\rho_1}\max(x^{\epsilon_0},x^{-\epsilon_0})c(x,y)\,\dd x^{-\gamma_1}\right|
\end{split}
\end{equation}
Furthermore, using the triangular inequality and Cauchy-Schwartz, we get
\[
\begin{split}
&\left|\int_0^{l_n} x^{-\rho_1}\max(x^{\epsilon_0},x^{-\epsilon_0})c(x,y)\,\dd x^{-\gamma_1}\right|\\
&\leq \int_0^{1} x^{-\rho_1-\epsilon_0}c(x,y)\,\dd x^{-\gamma_1}+\left|\int_1^\infty c(x,y)^2\,\dd x^{-\gamma_1}\right|^{1/2}\left|\int_1^{l_n} x^{-2\rho_1+2\epsilon_0}\,\dd x^{-\gamma_1}\right|^{1/2}\\
&=O\left(l_n^{-\rho_1+\epsilon_0-\frac{\gamma_1}{2}} \right)
\end{split}
\]
Going back to~\eqref{eqn:iii_1}, we obtain
\[
\sup_{y\in[1/2,2]}T_n\left|\int_0^{l_n} \left|\left(\frac{s_n(x)}{x}\right)^{1/\eta}-1\right|c(x,y)\,\dd x^{-\gamma_1}\right|=O\left(\sqrt{k}\left(\frac{n}{k}\right)^{-\frac{1}{2\eta}+\frac{1}{2}-\lambda\frac{\gamma_1}{2}} \right)\to 0,
\]
because of assumption $B(5)$.

Next, we deal with the second term in the right hand side of \eqref{eqn:iii}. By Cauchy-Schwartz and assumption $B(1)$, we obtain
\begin{equation}
\label{eqn:iii}
\begin{split}
\left|\int_{l_n}^\infty c(x,y)\,\dd x^{-\gamma_1}\right|&\leq \gamma_1\left(\int_{l_n}^\infty x^{-\gamma_1-1}\,\dd x\right)^{1/2}\left(\int_{1}^\infty c(x,y)^2x^{-\gamma_1-1}\,\dd x\right)^{1/2}\\
&\leq C l_n^{-\gamma_1/2}
\end{split}
\end{equation}
for some constant $C>0$. Moreover, since 
\[
T_nl_n^{-\gamma_1/2}=\sqrt{k}\left(\frac{n}{k}\right)^{-\frac{1}{2\eta}+\frac{1}{2}-\lambda\frac{\gamma_1}{2}},
\]
by assumption $B(5)$, it follows that 
\[
\sup_{y}T_n\left|\int_{l_n}^\infty c(x,y)\,\dd x^{-\gamma_1}\right|\to 0.
\]
Furthermore, the triangular inequality yields
\[
\begin{split}
&\left|\int_{l_n}^\infty c(s_n(x),y)\,\dd x^{-\gamma_1}\right|\\
\leq &\left|\int_{l_n}^\infty c_{\frac{n}{k}}(s_n(x),y)\,\dd x^{-\gamma_1}\right|+\left|\int_{l_n}^\infty [c(s_n(x),y)-c_{\frac{n}{k}}(s_n(x),y)]\,\dd x^{-\gamma_1}\right|\\
\leq &\left|\int_{l_n}^\infty c_{\frac{n}{k}}(s_n(x),y)\,\dd x^{-\gamma_1}\right|+\sup_{\substack{0<x<n/k\\y\in[1/2,2]}}\frac{\left|c_{\frac{n}{k}}(x,y)-c(x,y)\right|}{x^{\beta_2}}\int_{l_n}^\infty s_n(x)^{\beta_2}\,\dd x^{-\gamma_1}.
\end{split}
\]
Note that, by assumption B(3),
\[
\sup_{\substack{0<x<n/k\\y\in[1/2,2]}}\frac{\left|c_{\frac{n}{k}}(x,y)-c(x,y)\right|}{x^{\beta_2}}=O_P\left(\left(\frac{n}{k}\right)^\tau\right).
\]
Then, using the definition of $s_n$, a change of variable and Jensen inequality, we obtain
\begin{equation}
\label{eqn:s_n^beta2}
\begin{split}
\int_{l_n}^\infty s_n(x)^{\beta_2}\,\dd x^{-\gamma_1}&=\int_{l_n}^\infty\left\{\frac{n}{k}\p\left(X>U_1(n/k)x^{-\gamma_1} \right) \right\}^{\beta_2}\,\dd x^{-\gamma_1}\\
&=\left(\frac{n}{k}\right)^{\beta_2}\int_0^{l_n^{-\gamma_1}}\left\{\p\left(X>U_1(n/k)x \right) \right\}^{\beta_2}\,\dd x\\
&\leq\left(\frac{n}{k}\right)^{\beta_2}l_n^{-\gamma_1}\left\{l_n^{\gamma_1}\int_0^{l_n^{-\gamma_1}}\p\left(X>U_1(n/k)x \right) \,\dd x\right\}^{\beta_2}\\
&=\left(\frac{n}{k}\right)^{\beta_2}l_n^{-\gamma_1}\left\{\frac{l_n^{\gamma_1}}{U_1(n/k)}\int_0^{U_1(n/k)l_n^{-\gamma_1}}\p\left(X>x \right) \,\dd x\right\}^{\beta_2}\\
&\leq \left(\frac{n}{k}\right)^{\beta_2-\lambda\gamma_1}\left(\frac{l_n^{\gamma_1}}{U_1(n/k)}\right)^{\beta_2}\E[X]^{\beta_2}\\
&=o\left(\left(\frac{n}{k}\right)^{\beta_2-\gamma_1}\right).
\end{split}
\end{equation}
Hence, assumption B(5) implies
\[
\begin{split}
&\sup_{y}T_n\left|\int_{l_n}^\infty [c(s_n(x),y)-c_{\frac{n}{k}}(s_n(x),y)]\,\dd x^{-\gamma_1}\right|\\
&\sqrt{k}\left(\frac{n}{k}\right)^{-\frac{1}{2\eta}+\frac{1}{2}+\tau+\beta_2-\lambda\gamma_1}\to 0.
\end{split}
\]
On the other hand, using the definition of $s_n$, we get
\[
\begin{split}
&\left|\int_{l_n}^\infty c_{\frac{n}{k}}(s_n(x),y)\,\dd x^{-\gamma_1}\right|\\
&\left|\int_{l_n}^\infty \left(\frac{n}{k}\right)^{1/\eta}\p\left[X>U_1\left(\frac{n}{ks_n(x)}\right),Y>U_2\left(\frac{n}{ky}\right)\right]\,\dd x^{-\gamma_1}\right|\\
&\leq \gamma_1\frac{ky}{n} \left(\frac{n}{k}\right)^{1/\eta}l_n^{-\gamma_1}
\end{split}
\]
and as a result
\begin{equation}
\label{eqn:iii-2}
\sup_{y}T_n\left|\int_{l_n}^\infty c_{\frac{n}{k}}(s_n(x),y)\,\dd x^{-\gamma_1}\right|\leq C \sqrt{k}\left(\frac{n}{k}\right)^{\frac{1}{2\eta}-\frac{1}{2}-\lambda\gamma_1}\to 0,
\end{equation}
because of assumption $B(5)$.

(iv) We write
\begin{equation}
\label{eqn:iv}
\begin{split}
&\sup_{y\in[1/2,2]}\left|\int_0^\infty c_{\frac{n}{k}}{\left(x\wedge\frac{n}{k},y\right)}\,\dd x^{-\rho\gamma_1}-\int_0^\infty c(x,y)\,\dd x^{-\rho\gamma_1}\right|\\
&\leq\sup_{y\in[1/2,2]}\left\{\left|\int_0^{\frac{n}{k}} c_{\frac{n}{k}}\left(x,y\right)\,\dd x^{-\rho\gamma_1}-\int_0^{\frac{n}{k}} c(x,y)\,\dd x^{-\rho\gamma_1}\right|
+\left|c_{\frac{n}{k}}\left(n/k,y\right)\int_{\frac{n}{k}}^\infty \,\dd x^{-\rho\gamma_1}\right|+\left|\int_{\frac{n}{k}}^\infty c(x,y)\,\dd x^{-\rho\gamma_1}\right|\right\}\\
&\leq  \sup_{\substack{0<x<n/k\\y\in[1/2,2]}} \frac{\left|c_{\frac{n}{k}}(x,y)-c(x,y) \right|}{x^{\beta_1}\wedge x^{\beta_2}}\left|\int_0^\infty x^{\beta_1}\wedge x^{\beta_2}\,\dd x^{-\rho\gamma_1}\right|{+\left( \frac{n}{k}\right)^{-\rho\gamma_1}c_{\frac{n}{k}}\left(n/k,2\right)+o(1)}.
\end{split}
\end{equation}
{The first term in the} right hand side of the inequality converges to zero by assumptions B(2)-B(3). {Moreover, by assumption B(1), we have
\[
\left( \frac{n}{k}\right)^{-\rho\gamma_1}c_{\frac{n}{k}}\left(n/k,2\right)\leq \left( \frac{n}{k}\right)^{\frac{1}{\eta}-1-\rho\gamma_1}\to 0.
\]}
(v) We first give the proof for \eqref{eq: cnc2}. By Assumptions B(2) and B(3), we have
\[
\begin{split}
&\sup_{y\in[1/2,2]}{T_n}\left|\int_0^\infty c_{\frac{n}{k}}({s_n(x)}, y)-c({s_n(x)}, y) \,\dd x^{-\gamma_1} \right|\\
&\leq T_n   \sup_{\substack{0<x<n/k\\y\in[1/2,2]}}  \frac{\left|c_{\frac{n}{k}}(x,y)-c(x,y)\right|}{x^{\beta_1}\wedge x^{\beta_2}}\int_0^\infty s_n(x)^{\beta_1}\wedge s_n(x)^{\beta_2}\,\dd x^{-\gamma_1}.\\
&=O\left(T_n\left(\frac{n}{k}\right)^\tau\right)\int_0^\infty s_n(x)^{\beta_1}\wedge s_n(x)^{\beta_2}\,\dd x^{-\gamma_1}.
\end{split}
\]
Next we obtain an upper bound for the integral in the last equality.  Since $s_n(x)$ is monotone and $s_n(1)=1$, we get the following bound for the integral from zero to one,
\[
\int_0^1 s_n(x)^{\beta_1}\,\dd x^{-\gamma_1}<\int_{\R}s_n(x)^{\beta_1}\wedge 1\,\dd x^{-\gamma_1},
\]
which is shown to be $O(1)$  in \citep{Caietal2015} (page 438). Moreover, using the definition of $s_n$, a change of variable and Jensen inequality, we obtain
\[
\begin{split}
\int_1^\infty s_n(x)^{\beta_2}\,\dd x^{-\gamma_1}&=\int_1^\infty\left\{\frac{n}{k}\p\left(X>U_1(n/k)x^{-\gamma_1} \right) \right\}^{\beta_2}\,\dd x^{-\gamma_1}\\
&=\left(\frac{n}{k}\right)^{\beta_2}\int_0^1\left\{\p\left(X>U_1(n/k)x \right) \right\}^{\beta_2}\,\dd x\\
&\leq\left(\frac{n}{k}\right)^{\beta_2}\left\{\int_0^1\p\left(X>U_1(n/k)x \right) \,\dd x\right\}^{\beta_2}\\
&=\left(\frac{n}{k}\right)^{\beta_2}\left\{\frac{1}{U_1(n/k)}\int_0^{U_1(n/k)}\p\left(X>x \right) \,\dd x\right\}^{\beta_2}\\
&\approx\left(\frac{n}{k}\right)^{\beta_2-\gamma_1\beta_2}\E[X].
\end{split}
\]
By Assumption B(5), 
\begin{equation}
\label{eqn:v}
O\left(T_n\left(\frac{n}{k}\right)^{\tau+\beta_2-\gamma_1\beta_2}\right)=\sqrt{k}\left( \frac{n}{k}\right)^{-\frac{1}{2\eta}+\frac{1}{2}+\tau+\beta_2(1-\gamma_1)}\to 0.
\end{equation}
Thus, \eqref{eq: cnc2} is proved.

{The proof for \eqref{eq: cnc1} can be obtained in a similar way. We use the triangular inequality as in~\eqref{eqn:iv} to get
\[
\begin{split}
&\sup_{y\in[1/2,2]}T_n\left|\int_0^\infty c_{\frac{n}{k}}\left(x\wedge\frac{n}{k},y\right)\,\dd x^{-\gamma_1}-\int_0^\infty c(x,y)\,\dd x^{-\gamma_1}\right|\\
&\leq T_n \sup_{\substack{0<x<n/k\\y\in[1/2,2]}} \frac{\left|c_{\frac{n}{k}}(x,y)-c(x,y) \right|}{x^{\beta_1}\wedge x^{\beta_2}}\left|\int_0^\infty x^{\beta_1}\wedge x^{\beta_2}\,\dd x^{-\rho\gamma_1}\right|\\
&~~~+T_n\left( \frac{n}{k}\right)^{-\gamma_1}c_{\frac{n}{k}}\left(n/k,2\right)+T_n\left|\int_{\frac{n}{k}}^\infty c(x,2)\,\dd x^{-\rho\gamma_1}\right|\\
&=A_1+A_2+A_3.
\end{split}
\]
$A_1$ converges to zero by~\eqref{eqn:v}.  Moreover, as in~\eqref{eqn:iii}, 
\[
A_3=O\left(T_n\left(\frac{n}{k}\right)^{-\gamma_1/2}\right)\to 0
\] 
by assumption B(5). Finally, 
\[
A_2= O_P\left(T_n \left(\frac{n}{k}\right)^{-\gamma_1+\frac{1}{\eta}-1}\right)=O_P\left(\sqrt{k}\left(\frac{n}{k}\right)^{\frac{1}{2\eta}-\frac{1}{2}-\gamma_1}\right)\to 0,
\]
(see~\eqref{eqn:iii-2}).}

\end{proof}
\begin{lemma}
\label{le:A_n}
Assume B(1)-B(3). For $y\in[1/2,2]$ and $\rho\in\{1,2,2+\bar{\delta}\}$, define 
\[
A_n(y,\rho)=\left(\frac{n}{k}\right)^{1/\eta}\left({-}\int_0^\infty \1_{\{1-F_1(X_1)<\frac{k}{n}x, 1-F_2(Y_1)<\frac{ky}{n}\}}\,\dd x^{-\gamma_1}\right)^{\rho}.
\]
Then
\[
\E[A_n(y,\rho)]\to -\int_0^\infty c(x,y)\,\dd x^{-\rho\gamma_1}.
\] 
\end{lemma}
\begin{proof}
Denote $W_1=1-F_1(X_1)$ and $V_1=1-F_2(Y_1)$. Then,
we can write the integral as
\[
\int_0^\infty \1_{\{{W}_1<\frac{k}{n}x,V_1<\frac{ky}{n}\}}\,\dd x^{-\gamma_1}=-\1_{\{V_1<\frac{ky}{n}\}}\left(\frac{n}{k}\comR{W}_1 \right)^{-\gamma_1}.
\]
As a result, by~\eqref{eqn:expectation} and a change of variable, we obtain
\[
\begin{split}
\E[A_n(y,\rho)]&=\left(\frac{n}{k}\right)^{1/\eta}\E\left[\1_{\{V_1<\frac{ky}{n}\}}\left(\frac{n}{k}{W}_1 \right)^{-\rho\gamma_1}\right]\\
&=\left(\frac{n}{k}\right)^{1/\eta}\int_0^\infty \p\left[{W}_1<\frac{k}{n}x^{-\frac{1}{\rho\gamma_1}},V_1<\frac{ky}{n}\right]\,\dd x\\
&=-\left(\frac{n}{k}\right)^{1/\eta}\int_0^\infty \p\left[{W}_1<\frac{k}{n}x,V_1<\frac{ky}{n}\right]\,\dd x^{-\rho\gamma_1}\\
&=-\int_0^\infty c_{\frac{n}{k}}{\left(x\wedge\frac{n}{k},y\right)}\,\dd x^{-\rho\gamma_1}.
\end{split}
\]
The statement follows from \eqref{eqn:limit_c_n}.
\end{proof}

\begin{proof}[Proof of Proposition \ref{prop:limit_process}]
For $i=1,\dots,n$, let $W_i=1-F_1(X_i)$ and $V_i=1-F_2(Y_i)$. We write
\[
\begin{split}
\tilde{\theta}_{\frac{ky}{n}}&=\frac{1}{ky}\sum_{i=1}^n\int_0^\infty \1_{\{X_i>x,Y_i>U_2\left(\frac{n}{ky}\right)\}}\,\dd x\\
&=-\frac{U_1(n/k)}{ky}\sum_{i=1}^n\int_0^\infty \1_{\{W_i<\frac{k}{n}s_n(x),V_i<\frac{k}{n}y\}}\,\dd x^{-\gamma_1}
\end{split}
\]
where $s_n(x)$ is defined in \eqref{eq: sn}.

Similarly, we have
\begin{equation}
\label{eqn:theta}
\begin{split}
\theta_{\frac{ky}{n}}=
-\frac{nU_1(n/k)}{ky}\int_0^\infty \p\left(W_1<\frac{k}{n}s_n(x),V_1<\frac{k}{n}y\right)\,\dd x^{-\gamma_1}.
\end{split}
\end{equation}
This means that $\E[\tilde{\theta}_{\frac{ky}{n}}]={\theta}_{\frac{ky}{n}}$, and it enables us to write the left hand side of \eqref{eq: limit_process} as 
\begin{equation*}
\frac{S_n}{U_1(n/k)}\left(\tilde{\theta}_{\frac{ky}{n}}-\theta_{\frac{ky}{n}}\right)
=\sum_{i=1}^n \left(Z^*_{n,i}(y)-\E[Z^*_{n,i}(y)]\right),
\end{equation*}
where
\begin{equation}
\label{eqn:Z*}
Z^*_{n,i}(y)=-\frac{S_n}{ky}\int_0^\infty \1_{\left\{W_i<\frac{k}{n}s_n(x),V_i<\frac{k}{n}y\right\}}\,\dd x^{-\gamma_1}.
\end{equation}

 Recall that we have $\lim_{n\rightarrow\infty}s_n(x)=x$ by the regular variation of $1-F_1$. We shall study a simpler process obtained by replacing $s_n(x)$ with $x$ in \eqref{eqn:Z*}:
\begin{equation}
\label{eqn:Z}
Z_{n,i}(y)=-\frac{S_n}{ky}\int_0^\infty \1_{\left\{W_i<\frac{k}{n}x,V_i<\frac{k}{n}y\right\}}\,\dd x^{-\gamma_1}.
\end{equation}
To prove \eqref{eq: limit_process}, it suffices to show that 
\begin{equation}
\sup_{y\in[1/2,2]}n\E\left[|Z^*_{n,1}(y)-Z_{n,1}(y)|\right]\to 0,  \label{eq: approx}
\end{equation}
and
\begin{equation}
\left\{\sum_{i=1}^n \left(Z_{n,i}(y)-\E[Z_{n,i}(y)]\right)\right\}_{y\in[1/2, 2]}
\overset{d}{\to}\left\{W(y)\right\}_{y\in[1/2, 2]}. \label{eq: limit_process2}
\end{equation}
Note that \eqref{eq: approx} implies that 
\[
\sup_{y\in[1/2,2]}\sum_{i=1}^n\left(Z^*_{n,i}(y)-Z_{n,i}(y)\right)\xrightarrow{\p}0
\]
and
\[
\sup_{y\in[1/2,2]}\sum_{i=1}^n\left(\E[Z^*_{n,i}](y)-\E[Z_{n,i}(y)]\right)\xrightarrow{\p}0
\]

\noindent{\bf Step 1: Proof of \eqref{eq: approx}}\\
Using the definitions and the triangular inequality we write
\[
\begin{split}
n\E\left[|Z^*_{n,1}-Z_{n,1}|\right]&=-\left(\frac{n}{k}\right)^{\frac{1}{2\eta}-\frac{1}{2}}\frac{n}{\sqrt{k}y}\int_0^\infty \p\left(\frac{k}{n}\left(x\wedge s_n(x)\right)< {W}_1<\frac{k}{n}\left(x\vee s_n(x)\right),V_1<\frac{k}{n}y\right)\,\dd x^{-\gamma_1}\\
&=-\frac{T_n}{y}\int_0^\infty \left[c_{\frac{n}{k}} {\left(\left(x\wedge\frac{n}{k}\right)\vee s_n(x),y\right)}-c_{\frac{n}{k}}(x\wedge s_n(x),y) \right]\,\dd x^{-\gamma_1}\\
&\leq-\frac{T_n}{y}\int_0^\infty \left|c(x,y)-c(s_n(x),y) \right|\,\dd x^{-\gamma_1}\\
&\quad-\frac{T_n}{y}\int_0^\infty \left|c_{\frac{n}{k}} {\left(x\wedge\frac{n}{k},y\right)}-c(x,y) \right|\,\dd x^{-\gamma_1}\\
&\quad-\frac{T_n}{y}\int_0^\infty \left|c_{\frac{n}{k}}(s_n(x),y)-c(s_n(x),y) \right|\,\dd x^{-\gamma_1}.
\end{split}
\]
All three terms in the left hand side converge to zero by Lemma \ref{le:properties} (iii) and (v).

\noindent{\bf Step 2: Proof of \eqref{eq: limit_process2}}

We aim to apply Theorem 2.11.9 in \citep{VW96}. We will prove that the four conditions of this theorem are satisfied.
Here $(\F,\rho)=\{[1/2,2],\,\rho(y_1,y_2)=|y_1-y_2|\}$ and $\Vert Z\Vert_\F=\sup_{y\in\F}|Z(y)|$.

a) Fix $\epsilon>0$. Using that $\Vert Z_{n,1}\Vert_\F\leq 4 Z_{n,1}(2)$, we get, with $\bar\delta$ as defined in Assumption B(1).
\begin{align}
n\E\left[\Vert Z_{n,1}\Vert_\F\1_{\{\Vert Z_{n,1}\Vert_\F>\epsilon\}}\right]&\leq 4n\E\left[Z_{n,1}(2) \1_{\{ Z_{n,1}(2)>\epsilon\}}\right] \nonumber\\
&\leq \frac{4n}{ {\epsilon^{1+\bar\delta}}}\E\left[Z^{2+\bar\delta}_{n,1}(2) \right]\nonumber\\
&=\frac{n}{\epsilon^{1+\bar{\delta}}}\frac{S_n^{2+\bar\delta}}{2^{\bar\delta}k^{2+\bar\delta}}\E\left[\left( {-}\int_0^\infty \1_{\left\{W_i<\frac{k}{n}x,V_i<\frac{2k}{n}\right\}}\,\dd x^{-\gamma_1}\right)^{2+\bar\delta}\right]\nonumber\\
&=\frac{1}{ {\epsilon^{1+\bar{\delta}}}2^{\bar\delta}}T_n^{- {\bar\delta}}\E\left[\left(\frac{n}{k}\right)^{1/\eta}\left( {-}\int_0^\infty \1_{\left\{W_i<\frac{k}{n}x,V_i<\frac{2k}{n}\right\}}\,\dd x^{-\gamma_1}\right)^{2+\bar\delta}\right]\nonumber\\
&\to 0. \label{eq: lin1}
\end{align}
The last convergence follows from that $T_n\to \infty$ and Lemma \ref{le:A_n}.

b) Take a sequence  $\delta_n\to 0$. Then, by the triangular inequality and that $y_1, y_2\geq 1/2$, it follows that 
\begin{equation}
\label{eq: Zni}
\begin{split}
&\sup_{|y_1-y_2|<\delta_n}\sum_{i=1}^n\E\left[(Z_{n,i}(y_1)-Z_{n,i}(y_2))^2 \right]\\
&\leq 4\sup_{|y_1-y_2|<\delta_n}\sum_{i=1}^n\left(\frac{n}{k} \right)^{\frac{1}{\eta}-1}\frac{1}{k}\E\left[\left(\int_0^\infty \1_{\{ {W}_i<\frac{k}{n}x, \frac{k}{n}y_2<V_i<\frac{k}{n}y_1\}}\,\dd x^{-\gamma_1}\right)^2 \right]\\
&\quad+\sup_{|y_1-y_2|<\delta_n}\left|\frac{1}{y_2^2}-\frac{1}{y_1^2} \right|\E\left[\left(\frac{n}{k} \right)^{\frac{1}{\eta}}\left(\int_0^\infty \1_{\{ {W}_1<\frac{k}{n}x, V_1<\frac{k}{n}y_1\}}\,\dd x^{-\gamma_1}\right)^2 \right]\\
&= 4\sup_{|y_1-y_2|<\delta_n}\left(\frac{n}{k} \right)^{\frac{1}{\eta}}\E\left[\1_{\{\frac{k}{n}y_1<V_1<\frac{k}{n}y_2\}}\left(\frac{n}{k} {W}_1\right)^{-2\gamma_1}\right]
\quad+\sup_{|y_1-y_2|<\delta_n}\left|\frac{1}{y_2^2}-\frac{1}{y_1^2} \right|\E[A_n(y_1,2)],
\end{split}
\end{equation}
where $A_n(y_1,2)$  is defined as in Lemma \ref{le:A_n}. Thus, the second summand converges to zero $\lim_{n\to\infty}\E[A_n(y_1,2)]<\infty$ and $\delta_n\to 0$. 
 Moreover,  {by the triangular inequality and Lemma \ref{le:properties} (ii, iv)},
\begin{equation}
\begin{split}
&\left(\frac{n}{k} \right)^{\frac{1}{\eta}}\E\left[\1_{\{\frac{k}{n}y<V_1<\frac{k}{n}(y+\delta_n)\}}\left(\frac{n}{k} {W}_1\right)^{-2\gamma_1}\right]\\
&=-\left(\frac{n}{k} \right)^{\frac{1}{\eta}}\int_0^\infty \p\left(\frac{k}{n}y<V_1<\frac{k}{n}(y+\delta_n), {W}_1<\frac{k}{n}x\right)\,\dd x^{-2\gamma_1}\\
&= {\left|\int_0^\infty c_{\frac{n}{k}}\left(x\wedge\frac{n}{k},y+\delta_n\right)\,\dd x^{-2\gamma_1}-\int_0^\infty c_{\frac{n}{k}}\left(x\wedge\frac{n}{k},y\right)\,\dd x^{-2\gamma_1}\right|}\to 0.
\end{split}
\label{eq: Cn}
\end{equation}

c) Let $N_{[]}(\epsilon,\F,L_2^n)$ be the minimal number of sets $N_\epsilon$ in a partition $[1/2,2]=\cup_{j=1}^{N_\epsilon} I_{n,j}^\epsilon$ such that 
\[
\sum_{i=1}^n\E\left[\sup_{y_1,y_2\in I^\epsilon_{n,j}}|Z_{n,i}(y_1)-Z_{n,i}(y_2)|^2\right]\leq \epsilon^2,\qquad\forall\,j=1,\dots,N_{\epsilon}.
\] 
Consider the partition given by $I_{n,j}^\epsilon=[1/2+(j-1)\Delta_n,1/2+j\Delta_n]$. Then $N_\epsilon=3/2\Delta_n$. We aim at finding $\Delta_n$ such that for every sequence $\delta_n\to 0$ it holds
\[
\int_0^{\delta_n}\sqrt{\log N_{[]}(\epsilon,\F,L_2^n) }\,\dd \epsilon\to 0.  
\] 
By the same reasoning for \eqref{eq: Zni}, we obtain
\[
\begin{split}
&n\E\left[\sup_{y_1, y_2\in I^\epsilon_{n,j}}|Z_{n,1}(y_1)-Z_{n,1}(y_2)|^2\right]\\
\leq &\sup_{y_1, y_2\in I^\epsilon_{n,j}}\left|\frac{1}{y_1^2}-\frac{1}{y_2^2}\right| \E A_n(y_1,2)
+4\sup_{y_1, y_2\in I^\epsilon_{n,j}}\left(\frac{n}{k} \right)^{\frac{1}{\eta}}\E\left[\1_{\{\frac{k}{n}y_1<V_1<\frac{k}{n}y_2\}}\left(\frac{n}{k} {W}_1\right)^{-2\gamma_1}\right]\\
=:&B_n+C_n.
\end{split}
\]
For the first term we have $B_n\leq K_1\Delta_n$ for some constant $K_1>0$ by Lemma \ref{le:A_n}. Let $\bar{y}_1=1/2+(j-1)\Delta_n$ and $\bar{y}_2=1/2+j\Delta_n$. Next, we derive two different upper bounds for $C_n$. First, by Holder inequality, we obtain 
\[
\begin{split}
C_n&\leq 4\left(\frac{n}{k}\right)^{\frac{1}{\eta}}\E\left[ \1_{\{\frac{k}{n}\bar{y}_1<V_1<\frac{k}{n}\bar{y}_2 \}}\left( {W}_1\frac{n}{k} \right)^{-2\gamma_1}\right]\\
&\leq 4\left(\frac{n}{k}\right)^{\frac{1}{\eta}}\E\left[ \1_{\{\frac{k}{n}\bar{y}_1<V_1<\frac{k}{n}\bar{y}_2 \}}\right]^{1/p}\E\left[ \1_{\{V_1<\frac{k}{n}\bar{y}_2 \}}\left( {W}_1\frac{n}{k} \right)^{-2q\gamma_1}\right]^{1/q}\\
&\leq 4\left(\frac{n}{k}\right)^{\frac{1}{\eta}-\frac{1}{p}-\frac{1}{\eta q}}|\bar{y}_1-\bar{y}_2|^{\frac{1}{p}}\E[A_n(\bar{y}_2,2q)]\\
&=K_2 \left(\frac{n}{k}\right)^{\frac{1}{\eta}-\frac{1}{p}-\frac{1}{\eta q}}\Delta_n^{\frac{1}{p}},
\end{split}
\]
for some constant $K_2$. The last equality is obtained by applying Lemma \ref{le:A_n} and choosing  $q=(2+\bar{\delta})/2$ and $1/p+1/q=1$.

Second,  by the same reasoning for \eqref{eq: Cn}, the triangular inequality and Lemma \ref{le:properties} (ii), (iv), we get a second bound on $C_n$,
\[
\begin{split}
C_n&\leq -8\int_0^\infty \left[c_{\frac{n}{k}} {\left(x\wedge\frac{n}{k},\bar{y}_2\right)}-c_{\frac{n}{k}} {\left(x\wedge\frac{n}{k},\bar{y}_1\right)}\right]\,\dd x^{-2\gamma_1}\\
&=-8\int_0^\infty [c(x,\bar{y}_2)-c(x,\bar{y}_1]\,\dd x^{-2\gamma_1}\\
&\quad-8\int_0^\infty \left[c_{\frac{n}{k}} {\left(x\wedge\frac{n}{k},\bar{y}_2\right)}-c(x,\bar{y}_2)\right]\,\dd x^{-2\gamma_1}\\
&\quad-8\int_0^\infty \left[c(x,\bar{y}_1)-c_{\frac{n}{k}} {\left(x\wedge\frac{n}{k},\bar{y}_1\right)}\right]\,\dd x^{-2\gamma_1}\\
&\leq K_3\Delta_n+K_4\left(\frac{n}{k}\right)^\tau
\end{split}
\]
for some constants $K_3$ and $K_4$. 

If $\epsilon^2<\left(\frac{n}{k}\right)^{\tau^*}$ for some $\tau^*\in(\tau,0)$  we use the first bound on $C_n$, i.e.  
\[
B_n+C_n\leq 2K_2 \left(\frac{n}{k}\right)^{\frac{1}{\eta}-\frac{1}{p}-\frac{1}{\eta q}}\Delta_n^{\frac{1}{p}}
\]
and by choosing 
\[
\Delta_n=(2K_2)^{-p}\left(\frac{n}{k}\right)^{-\frac{p}{\eta}+1+\frac{p}{\eta q}}\epsilon^{2p},
\]
we get $B_n+C_n\leq \epsilon^2$. Hence 
\[
N_\epsilon\leq\frac{ {3(2K_2)^p}}{\epsilon^{2p}}\left(\frac{n}{k}\right)^{\frac{p}{\eta}-1-\frac{p}{\eta q}}.
\]
Otherwise, if $\epsilon^2>\left(\frac{n}{k}\right)^{\tau^*}$, for sufficiently large $n$,  
\[
K_4\left(\frac{n}{k}\right)^\tau<\frac{1}{2}\left(\frac{n}{k}\right)^{\tau^*}<\frac{1}{2}\epsilon^2
\]
and we use the second bound on $C_n$ with 
\[
\Delta_n=\frac{\epsilon^2}{2(K_1+K_3)},
\]
i.e. we get
\[
B_n+C_n\leq (K_1+K_3)\Delta_n+K_4 \left(\frac{n}{k}\right)^\tau\leq \epsilon^2. 
\]
Hence, in this case,  
\[
N_\epsilon\leq\frac{3 {(K_1+K_3)}}{\epsilon^{2}}.
\]

Now we distinguish between two cases. If $\delta_n\sqrt{\log (n/k)}\to 0$, using $\sqrt{a+b}\leq\sqrt{a}+\sqrt{b}$ and $\log(x)\leq x$ for large $x$, we get
\[
\begin{split}
\int_0^{\delta_n}\sqrt{\log N_{[]}(\epsilon,\F,L_2^n)}\,\dd \epsilon &\leq \int_0^{\delta_n}\sqrt{\left( \frac{p}{\eta}-1-\frac{p}{\eta q}\right)\log (n/k)+2p\log \epsilon^{-1}+\log  {3(2K_2)^p}}\,\dd \epsilon\\
&\leq K\left(\int_0^{\delta_n}\sqrt{\log (n/k)}\,\dd \epsilon+\int_0^{\delta_n}\sqrt{\epsilon^{-1}}\,\dd \epsilon\right)
\end{split}
\]
and the left hand side converges to zero as $\delta_n\to 0$.

On the other hand, if $\delta_n\sqrt{\log (n/k)}\nrightarrow 0$, take $\delta_n^*=(n/k)^{\tau^*}$. Note that $\delta^*_n\sqrt{\log (n/k)}\to 0$. Hence we write 
\[
\begin{split}
\int_0^{\delta_n}\sqrt{\log N_{[]}(\epsilon,\F,L_2^n)}\,\dd \epsilon &=\int_0^{\delta_n^*}\sqrt{\log N_{[]}(\epsilon,\F,L_2^n)}\,\dd \epsilon +\int_{\delta_n^*}^{\delta_n}\sqrt{\log N_{[]}(\epsilon,\F,L_2^n)}\,\dd \epsilon \\
\leq o(1)+\int_{\delta_n^*}^{\delta_n}\sqrt{\log ( {3(K_1+K_3)}/\epsilon^2)}\,\dd \epsilon\\
\leq o(1)+\sqrt{2}\int_{0}^{\delta_n}\sqrt{\epsilon^{-1}}\,\dd \epsilon\to 0.
\end{split}
\]

d) We have to show that the marginals converge, i.e.  {for each $M\in\N$ and }for each $y_1,\dots ,y_M\in[1/2,2]$, the random vector
\[
\left(\sum_{i=1}^n\left(Z_{n,i}(y_1)-\E[Z_{n,i}(y_1)]\right),\dots,\sum_{i=1}^n\left(Z_{n,i}(y_M)-\E[Z_{n,i}(y_M)]\right)\right)
\]
converges to a multivariate normal distribution. It suffices to show  that for each $a_1,\dots,a_M\in\R$ we have
\[
\sum_{j=1}^M a_j\left[ \sum_{i=1}^n\left(Z_{n,i}(y_j)-\E[Z_{n,i}(y_j)]\right)\right]=:\sum_{i=1}^n \left(N_{n,i}-\E[N_{n,i}]\right)
\]
converges a normal distribution, where $N_{n,i}=\sum_{j=1}^M a_j Z_{n,i}(y_j)$. This will follow from the  Lindeberg-Feller central limit theorem (see e.g. \citep{Vaart98}), once we show that for each $\epsilon>0$,
\begin{equation}
\label{eqn:cond1-CLT}
\sum_{i=1}^n\E\left[  | N_{n,i}|^2\1_{\{| N_{n,i}|>\epsilon\}}\right]\to 0
\end{equation}
and
\begin{equation}
\label{eqn:cond2-CLT}
\sum_{i=1}^n\mathrm{Var}\left(N_{n,i}\right)\to\sigma^2_N.
\end{equation}
We proceed with \eqref{eqn:cond1-CLT}. First,
\begin{align*}
\sum_{i=1}^n\E\left[  | N_{n,i}|^2\1_{\{| N_{n,i}|>\epsilon\}}\right]
&=n\E\left[  | N_{n,1}|^2\1_{\{| N_{n,1}|>\epsilon\}}\right]\\
&\leq \frac{n\E\left[  |N_{n,1}|^{2+\bar{\delta}}\right]}{\epsilon^{\bar{\delta}}}
\leq K n\sum_{j=1}^M |a_j|^{2+\bar{\delta}}\frac{\E\left[  |Z_{n,1}(2)|^{2+\bar{\delta}}\right]}{\epsilon^{\bar{\delta}}},
\end{align*}
which converges to zero by \eqref{eq: lin1}. For \eqref{eqn:cond2-CLT}, we write
\[
\begin{split}
\sum_{i=1}^n\mathrm{Var}(N_{n,i})&=n\left\{\E\left[\left( \sum_{j=1}^M a_jZ_{n,1}(y_j)\right)^2 \right]-\left(\E\left[\sum_{j=1}^M a_jZ_{n,1}(y_j) \right]\right)^2 \right\}\\
&=n\E\left[ \sum_{j,k=1}^M a_ja_k Z_{n,1}(y_j)Z_{n,1}(y_k)\right]-\left(\sqrt{n}\sum_{j=1}^M a_j\E\left[Z_{n,1}(y_j) \right]\right)^2 \\
&=n\sum_{j,k=1}^M a_ja_k\E\left[  Z_{n,1}(y_j)Z_{n,1}(y_k)\right]+o(1)
\end{split}
\]
because it is easy to check that $\sqrt{n}\E\left[Z_{n,1}(y_j) \right]\to 0$, for $j=1,\ldots, M$. Moreover, observe that 
\[
\begin{split}
n \E\left[  Z_{n,1}(y_j)Z_{n,1}(y_k)\right]=&\left(\frac{n}{k}\right)^{\frac{1}{\eta}}\frac{1}{y_jy_k}  {\E\left[\left(\int_0^\infty \1_{\{U_1<\frac{k}{n}x,V_1<\frac{k}{n}(y_j\wedge y_k)\}}\,\dd x^{-\gamma_1}\right)^{2}\right]}\\
=&\frac{1}{y_jy_k} {\E[A_n(y_j\wedge y_k,2)]}.
\end{split}
\]
Thus, by Lemma \ref{le:A_n}, it follows that \eqref{eqn:cond2-CLT} holds with
\begin{equation}
\label{eq: var}
\sigma^2_N=-\sum_{j,k=1}^M\frac{ a_ja_k}{y_jy_k}\int_0^\infty c(x,y_j\wedge y_k)\,\dd x^{-2\gamma_1}.
\end{equation}

We have now verified the four conditions required by Theorem 2.11.9 in \citep{VW96}, which leads to the conclusion that $\sum_{i=1}^n(Z_{n,i}-\E[Z_{n,i}])$ converges in distribution to a Gaussian process. Finally, we compute the covariance structure of the limit process. For each $y_1,\,y_2\in[1/2,2],$ by independence, we have 
\[
\begin{split}
\E[W(y_1)W(y_2)]&=\lim_{n\to\infty} \mathrm{Cov}\left( \sum_{i=1}^n Z_{n,i}(y_1), \sum_{i=1}^n Z_{n,i}(y_2)\right)\\
&=\lim_{n\to\infty}n \mathrm{Cov}(Z_{n,1}(y_1), Z_{n,1}(y_2))\\
&=\lim_{n\to\infty}\left(n\E\left[  Z_{n,1}(y_1)Z_{n,1}(y_2)\right]-n\E\left[  Z_{n,1}(y_1)\right]\E\left[  Z_{n,1}(y_2)\right]\right)\\
&= -\frac{1}{y_1y_2}\int_0^\infty c(x,y_1\wedge y_2)\,\dd x^{-2\gamma_1}\\
&=\frac{1}{y_1y_2}\int_0^\infty c\left(x^{-\frac{1}{2\gamma_1}},y_1\wedge y_2\right)\,\dd x.
\end{split}
\]
The fourth equality follows the same reasoning as that for \eqref{eq: var}.

\end{proof}

  \begin{proof}[Proof of Proposition \ref{prop:as.distr_theta}]
Note the convergence speed in this proposition is $\frac{S_n}{U_1(n/k)}$, the same as that in Proposition \ref{prop:limit_process}. By definition 
\[
\hat{\theta}_{\frac{k}{n}}=e_n\tilde{\theta}_{\frac{ke_n}{n}}.
\]
Hence
\begin{equation}
\label{eqn:1}
\begin{split}
\frac{S_n}{U_1(n/k)}\left(\hat{\theta}_{\frac{k}{n}}-\theta_{\frac{k}{n}} \right)&=\frac{S_n}{U_1(n/k)}e_n\left(\tilde{\theta}_{\frac{ke_n}{n}}-\theta_{\frac{ke_n}{n}} \right)+\frac{S_n}{U_1(n/k)}\left(e_n\theta_{\frac{ke_n}{n}}-\theta_{\frac{k}{n}} \right)\\
&=:T_1+T_2.
\end{split}
\end{equation}
First, we show that $T_1\overset{d}{\rightarrow} W(1)$. We start by writing
\[
T_1=e_n\left[\frac{{S_n}}{U_1(n/k)}(\tilde{\theta}_{\frac{ke_n}{n}}-\theta_{\frac{ke_n}{n}})-W(e_n) \right]+e_nW(e_n).
\]
Because $e_n\xrightarrow{\p}1$, the first term of the right hand side, with probability tending to one, is bounded by 
\[
2\sup_{y\in[1/2,2]}\left|\frac{{S_n}}{U_1(n/k)}(\tilde{\theta}_{\frac{ky}{n}}-\theta_{\frac{ky}{n}})-W(y) \right|,
\] 
which is $o_p(1)$ by Proposition \ref{prop:limit_process} and continuous mapping theorem. Moreover, by Corollary 1.11 in \citep{Adler1990} and $e_n\xrightarrow{\p}1$, $W(e_n)\xrightarrow{d} W(1)$. Thus, $T_1\overset{d}{\rightarrow} W(1)$.

Using \eqref{eqn:theta},  we can write 
\[
\begin{split}
\frac{S_n}{U_1(n/k)}\theta_{\frac{ky}{n}}=-\frac{T_n}{y}\int_0^\infty c_{\frac{n}{k}}(s_n(x), y) \,\dd x^{-\gamma_1}.
\end{split}
\]

Thus, $T_2$ can be rewritten as follows
\[
\begin{split}
T_2&=T_n\int_0^\infty \left\{c_{\frac{n}{k}}(s_n(x),e_n)-c_{\frac{n}{k}}(s_n(x),1)\right\}\,\dd x^{-\gamma_1}\\
&=T_n\int_0^\infty \left\{c_{\frac{n}{k}}(s_n(x),e_n)-c(s_n(x),e_n)\right\}\,\dd x^{-\gamma_1}\\
&\quad +T_n\int_0^\infty \left\{c(s_n(x),e_n)-c(s_n(x),1)\right\}\,\dd x^{-\gamma_1}\\
&\quad+T_n\int_0^\infty \left\{c(s_n(x),1)-c_{\frac{n}{k}}(s_n(x),1)\right\}\,\dd x^{-\gamma_1}\\
&=o_p(1)+T_{21}+o(1).
\end{split}
\]
The last equality follows from the fact that $e_n\overset{P}{\to}1$ and \eqref{eq: cnc2}.
Further, we can decompose $T_{21}$ into three terms as follows.
\[
\begin{split}
T_{21}&=T_n\int_0^\infty \left\{c(s_n(x),e_n)-c(x,e_n)\right\}\,\dd x^{-\gamma_1}+T_n\int_0^\infty \left\{c(x,1)-c(s_n(x),1)\right\}\,\dd x^{-\gamma_1}\\
&\quad +T_n\int_0^\infty \left\{c(x,e_n)-c(x,1)\right\}\,\dd x^{-\gamma_1}.\\
&=o_p(1)+o(1)+O_P(T_n|e_n-1|),
\end{split}
\]
by Lemma \ref{le:properties} (iii) and (ii).
Finally,
\[
T_n|e_n-1| =\left(\frac{n}{k}\right)^{-\frac{1}{2\eta}+\frac{1}{2}}\sqrt{k}|e_n-1|\to 0,
\]
because $\sqrt{k}|e_n-1|=O_P(1)$ (see (26) in \citep{Caietal2015}). Consequently, $T_2=o_p(1)$ and it  has no contribution in the limit distribution.

\end{proof}

\bibliographystyle{plainnat}
\bibliography{EVT}

\end{document}